\documentclass[12pt]{article}
\usepackage{amsmath,amssymb,amsthm}
\usepackage{setspace}%
\usepackage[paperwidth=5.5in,paperheight=8.5in,rmargin=-0.1in,lmargin=-0.1in,tmargin=0.30in,bmargin=0.45in]{geometry}
\usepackage[utf8]{inputenc}
\usepackage[affil-it]{authblk}
\newcommand{\R}{\mathbb R}  
\newcommand{\Z}{\mathbb Z}
\newcommand{\N}{\mathbb N}

\newtheoremstyle{theoremstyle}{}{}{\itshape}{}{\scshape}{.}{ }{\textbf{#1\ #2}}  
\theoremstyle{theoremstyle}
\newtheorem{theorem}{Theorem}

\newtheorem{prop}{Proposition}

\newtheorem{conj}{Conjecture}

\begin{document}
\title{\textbf{Coloring the Real Line with Monochromatic Intervals}}
\author{\LARGE{Doyon Kim}
\\ \large{\textnormal{Auburn University}}
\\ \large\textnormal{dzk0028@auburn.edu}}

\maketitle 
\abstract {Suppose \(D\subseteq(0,\infty)\) and \(0<|D|<\infty\). The distance graph \(G(\R,D)\) is the graph with vertex set \(\R\), and two vertices \(x,y\) are adjacent if \(|x-y|\in D\). We prove that for every positive integer \(t>1\) there is a distance set \(D\) such that the chromatic number of \(G(\R,D)\) is \(t\) and no proper coloring of \(G(\R,D)\) with \(t\) colors allows monochromatic intervals. This result disproves a conjecture in [2].}
\section{Introduction}
\ \ \ \ Suppose that \(D\subseteq(0,\infty)\) and \(0<|D|<\infty\). What is the smallest number of colors needed to color \(\R\), so that the distances in \(D\) are \textit{forbidden}, meaning that if \(|x-y|\in D\) then \(x\) and \(y\) have different colors? This is a graph coloring problem. Let \(G(\R,D)\) denote the graph with vertex set \(\R\), such that any two vertices \(x,y\) are adjacent if and only if \(|x-y|\in D\). A coloring of the graph is \textit{proper} if no two adjacent vertices have the same color. The chromatic number of \(G(\R,D)\), denoted by \(\chi(G(\R,D))\), is the smallest number of colors needed to color \(\R\) properly. The problems of finding chromatic number of \(G(\R,D)\) with a given distance set \(D\) were introduced by Eggleton et al. [1].\par
Here, we impose one more condition. Given \(D\subseteq(0,\infty)\) and \(0<|D|<\infty\), what is the smallest number of colors needed to color \(\R\) \textit{with monochromatic intervals}, so that the distances in \(D\) are forbidden? A coloring \(\varphi:\R\to\{C_0,\dots,C_{t-1}\}\) is a \textit{slab} coloring if and only if for each \(j\in\{0,\dots,t-1\}\), \(\varphi^{-1}(C_j)=\{x\in\R\mid \varphi(x)=C_j\}\) is a union of intervals. So the question is: Given a finite distance set \(D\), what is the smallest number \(t\) such that there is a proper slab coloring \(\varphi:\R\to \{C_0,\dots,C_{t-1}\}\)? \par
\section{Coloring \(\R\) with monochromatic intervals} 
\ \ \ \ Let \(\chi_m(G(\R,D))\) denote the smallest number of colors needed to color \(\R\) with monochromatic intervals so that the distances in \(D\) are forbidden. Obviously, \(\chi_m(G(\R,D))\geq \chi(G(\R,D))\).
It is not difficult to prove that if the elements of \(D\) are commensurable then \(\chi_m(G(\R,D))=\chi(G(\R,D))\). Suppose that the elements of \(D\) are commensurable. Then there exists a positive number \(\alpha\) such that \(D'=\{\alpha d|d\in D\}\) is a set of positive integers. \(G(\R,D')\) is isomorphic to \(G(\R,D)\) so \(\chi(G(\R,D'))=\chi(G(\R,D))\). \par 
Also, \(\chi(G(\R,D'))=\chi(G(\Z,D'))\). To see this, first note that \(\Z\subseteq \R\) implies that \(\chi(G(\Z,D'))\leq \chi(G(\R,D'))\). On the other hand, let \(\varphi\) be a proper coloring of \(G(\Z,D')\) with \(\chi(G(\Z,D'))\) colors. Color \(\R\) by assigning to each interval \([n,n+1)\) the color \(\varphi(n)\). This coloring clearly forbids the integer distances in \(D'\). Therefore, \(\chi(G(\R,D'))\leq\chi(G(\Z,D'))\). \par
The proper coloring of \(G(\R,D')\) suggested above is a slab coloring, and multiplication of \(\R\) by \(\frac{1}{\alpha}\), which is an isomorphism from \(G(\R,D')\) to \(G(\R,D)\), carries this coloring to a proper slab coloring of \(G(\R,D)\). Therefore,
\begin{align*}
\chi_m(G(\R,D))&\leq \chi(G(\Z,D'))=\chi(G(\R,D')) \\ &=\chi(G(\R,D))\leq\chi_m(G(\R,D)),
\end{align*}
so \(\chi(G(\R,D))=\chi_m(G(\R,D))\) if the elements of \(D\) are commensurable. \par
Now, suppose the elements of \(D\) is not commensurable. In this case, the problem of determining \(\chi_m(G(\R,D))\) is still open. Among other noticeable results, Anderson et al. [2] proved that if \(D=\{d_1,\dots,d_k\}\) and \(0<d_1<\cdots<d_k\) then \(\chi_m(G(\R,D))\leq \min(k+1,\lceil\frac{d_k}{d_1}\rceil+1)\). At the end of the paper [2], they included the following conjecture: 
\begin{conj}
If \(D\subseteq(0,\infty)\) and \(0<|D|<\infty\), then \(\chi_m(G(\R,D))>\chi(G(\R,D))\) if and only if the elements of \(D\) are incommensurable and \(\chi(G(\R,D))=2\).
\end{conj}
This turns out to be not true. We will present a counterexample in section 3, and end the paper with a proof of the following theorem, an extension of the counterexample:
\begin{theorem}
For every positive integer \(t\geq 2\), there exists a distance set \(D\) with \(D\subseteq(0,\infty)\) and \(0<|D|<\infty\) such that \(\chi_m(G(\R,D))>\chi(G(\R,D))=t\).
\end{theorem}
Throughout the presentation of a counterexample and the proof of the theorem, we allow the Axiom of Choice. Allowing the Axiom of Choice, the following proposition, introduced and proved in [2], will be used to verify the counterexample and prove the theorem:
\begin{prop}
For \(D=\{d_1,\dots,d_k\}\subset (0,\infty)\), let \(\Z[D]=\{a_1d_1+\cdots+a_kd_k\mid a_1,\dots,a_k\in\Z\}\). Then \(\chi(G(\R,D))=\chi(G(\Z[D],D))\).
\end{prop}
\begin{proof}
\(\Z[D]\) is the additive subgroup of \(\R\) generated by the elements of \(D\). Clearly, \(\chi(G(\Z[D],D))\leq\chi(G(\R,D))\). Let \(\Z[D]\) be colored with \(\chi(G(\Z[D],D))\) colors properly, and then color all of \(\R\) with that many colors by copying the coloring of \(\Z[D]\) on each coset of \(\Z[D]\) in \(\R\), by choosing a representative \(r\) of the coset and then coloring each element \(r+b\in r+\Z[D]\) by the color of \(b\) in \(\Z[D]\). Since \(\R\) is a disjoint union of the cosets of \(\Z[D]\) in \(\R\) and each distance in \(D\) is forbidden within each coset, we have a proper coloring of \(\R\) with \(\chi(G(\Z[D],D))\) colors. Therefore, \(\chi(G(\R,D))=\chi(G(\Z[D],D))\).
\end{proof}
\section{Counterexample to the conjecture}
\ \ \ \ Let \(D=\{1,2,\sqrt 2,2\sqrt 2,1+\sqrt 2\}\). Let \(G=G(\R,D)\), and \(G(\Z[D])=G(\Z[D],D))\). Here, \(\Z[D]=\{a+b\sqrt 2\mid a,b\in\Z\}\). By the proposition, we have \(\chi(G)=\chi(G(\Z[D])\). Since, in any proper coloring of \(G(\Z[D])\), the colors of \(0,1,2\) are different, \(\chi(G(\Z[D])\geq 3\). For \(a,b\in\Z\), color \(a+b\sqrt 2\) red if \(a+b\equiv 0\pmod 3\), green if \(a+b\equiv 1\pmod 3\), blue if \(a+b\equiv 2\pmod 3\). This is a proper coloring of \(G(\Z[D])\) with \(3\) colors, so we conclude by Proposition 1 that \(\chi(G)=\chi(G(\Z[D])=3\). It is well-known that \(\Z[D]\) is dense in \(\R\). Since each color set in the coloring of \(\Z[D]\) given just above contains one of \(3\Z[D]+k\), \(k=0,1,2\), it follows that each color in any coloring of \(G\) arising from this coloring of \(\Z[D]\), by Proposition 1's proof, is dense in \(\R\). Therefore, no such coloring is a slab coloring.
\par
Now, with the same distance set defined above, let \(\varphi:\R\to \{C_0,C_1,C_2\}\) be a proper three-coloring of \(G\). The set \(\{0,1,2\}\) induces \(K_3\), so \(0,1,2\) should have different colors. Without loss of generality, we can assume that \(\varphi (0)=C_0\), \(\varphi (1)=C_1\) and \(\varphi (2)=C_2\). Then the color of every integer is uniquely determined, since for every \(a\in\Z\), \(\{a,a+1,a+2\}\) induces \(K_3\). Now, take an arbitrary \(a\in \Z\). Since \(\{a-1,a,a+\sqrt 2\}\) induces \(K_3\), the color of \(a+\sqrt 2\) is uniquely determined. With the color of \(a,a+\sqrt 2\) uniquely determined, the color of every \(a+b\sqrt 2\), \(b\in\Z\) is uniquely determined. Since the choice of \(a\in\Z\) is arbitrary, the color of every \(a+b\sqrt 2\in\Z[D]\) is uniquely determined. Therefore, in this proper 3-coloring of \(G\), the color of every point in \(\Z[D]\) is uniquely determined, except for the color names. So we can regard the coloring of \(\Z[D]\) in \(\R\) with red, green, blue colors explained above as the unique proper 3-coloring of \(\Z[D]\), and conclude that there is no slab coloring of \(\R\) with three colors, because each color in any proper coloring of \(G(\Z[D])\) with \(3\) colors is dense in \(\Z[D]\), and \(\Z[D]\) itself is dense in \(\R\). Therefore \(\chi_m(G)>3\). \par
This suffices to disprove the conjecture, but we can easily show that \(\chi_m(G)=4\). Let \(C=\{C_0,C_1,C_2,C_3\}\) be a color set. For every \(n\in\Z\), color \([n,n+1)\) with \(C_i\), for \(n\equiv i\pmod 4\). This is a slab coloring that forbids every distance in \(D\). (Note that \(\lceil \frac{2\sqrt 2}{1}\rceil+1=4\). The slab coloring with \(4\) colors given here is illustrative of part of the proof in [2] of the general upper bound on \(\chi_m(G(\R,D))\) mentioned earlier.)
\section{Proof of the Theorem}
\begin{proof}[\unskip\nopunct]
\ \ \ \ Let \(\N\) denote the set of non-negative integers. For an arbitrary \(t\in \N\cap [2,\infty)\), let the distance set be \(D=\{a+b\sqrt 2\mid a,b\in\N, 1\leq a+b\leq t-1\}\). Clearly, \(D\subseteq (0,\infty)\) and \(0<|D|<\infty\). Let \(G=G(\R,D)\) and \(G(\Z[D])=G((\Z[D],D))\). Here, \(\Z[D]=\{a+b\sqrt 2\mid a,b\in\Z\}\). Since \((0,\dots,t-1)\) forms \(K_t\) in \(G\), \(\chi(G)\geq t\). Let \(C=\{C_0,\dots,C_{t-1}\}\), and color each \(a+b\sqrt 2\in \Z[D]\) with \(C_i\), for \(a+b\equiv i \pmod t\). This is a proper coloring of \(\Z[D]\) with \(t\) colors. Therefore \(\chi(G)=\chi(G(\Z[D]))=t\). Since all of the \(C_0\), \dots, \(C_{t-1}\) points in \(\Z[D]\) are dense in \(\R\), because \(\Z[D]\) is and each color class in \(\Z[D]\) contains a set \(t\Z[D]+k\), \(k\in\{0,\dots,t-1\}\), every open interval in \(\R\) contains all of the \(t\) colors. Therefore this t-coloring is not a slab coloring. \par
Now, suppose \(\varphi:\R\to \{C'_0,\dots,C'_{t-1}\}\) is a proper t-coloring of \(G\) with \(\varphi(i)=C'_i\) for \(i\in\{0,\dots,t-1\}\). Then the color of every integer is uniquely determined, since for every \(a\in\Z\), \(\{a,\dots,a+t-1\}\) induces \(K_t\). Take an arbitrary \(a\in\Z\); \(\{a-t+2,\dots,a,a+\sqrt 2\}\) induces \(K_t\), so the color of \(a+\sqrt 2\) is uniquely determined. Actually, for every \(j\in\{1,\dots,t-1\}\), \(\{a-t+j+1,\dots,a,a+\sqrt 2,\dots,a+j\sqrt 2\}\) induces \(K_t\), so the color of \(a+j\sqrt 2\) is uniquely determined (you can determine the color of \(a+\sqrt 2\) first, and then \(a+2\sqrt 2\), and so on). With the color of each element of \(\{a,a+\sqrt 2,\dots,a+(t-1)\sqrt 2)\}\) uniquely determined, the color of every \(a+b\sqrt 2\), \(b\in\Z\) is uniquely determined. Since the choice of \(a\in\Z\) was arbitrary, the color of every \(a+b\sqrt 2\in\Z[D]\) is uniquely determined, and it is straightforward to see that, except for the color names, the coloring \(\varphi\) on \(\Z[D]\) is the same as the coloring of \(\Z[D]\) given in the first paragraph of this proof. Therefore the color sets on \(\Z[D]\) are dense in \(\R\). Therefore there are no monochromatic intervals in this coloring, so it is certainly not a slab coloring. Therefore \(\chi_m(G)>\chi(G)=t\). This completes the proof. \\
\end{proof}

\section*{Acknowledgment}
\ \ \ \ This work was supported by NSF grant no. 1262930, and was completed during and after the 2015 summer Research Experience for Undergraduates in Algebra and Discrete Mathematics at Auburn University.

\end{document}